\documentclass{amsart}

\usepackage{amsmath, amsthm, amssymb, latexsym, graphics}

\newcommand{\fr}[1]{\mathfrak{#1}}
\newcommand{\ca}[1]{\mathcal{#1}}
\newcommand{\bk}{{{\boldsymbol{k}}}}
\newcommand{\ac}{\operatorname{ac}}

\newcommand{\af}{\mathbf{f}}
\newcommand{\ag}{\mathbf{g}}
\newcommand{\cha}{\operatorname{char}}
\newcommand{\val}{\operatorname{v}}
\newcommand{\res}{\operatorname{r}}
\def \ann{\operatorname{Ann}}

\def \bsigma{{\boldsymbol{\sigma}}}

\def \i{{\boldsymbol{i}}}
\def \j{{\boldsymbol{j}}}
\def \k{{\boldsymbol{k}}}
\def \l{{\boldsymbol{l}}}

\def \f{\operatorname{f}}

\newtheorem{theorem}{Theorem}[section]
\newtheorem{lemma}[theorem]{Lemma}

\newtheorem{corollary}[theorem]{Corollary}

\newtheorem{definition}[theorem]{Definition}

\newtheorem{remark}[theorem]{Remark}

\theoremstyle{definition}

\title{ 
Quantifier elimination for valued fields equipped with an automorphism}
\author{Salih Durhan} 
\address{Mathematics Research and Teaching Group, Middle East Technical University, Northern Cyprus Campus, Kalkanli Guzelyurt, TRNC, Mersin 10, Turkey}
\author{G\"onen\c{c} Onay}
\address{Department of Mathematics, Mimar Sinan Fine Arts University, Bomonti, Sisli 34380, Istanbul,Turkey}

\begin{document}
\bibliographystyle{plain}

\maketitle

\begin{section}{introduction}

A {\em difference field} is a field equipped with a distinguished automorphism. The concept has derived from the study of functional equations like $f(x+1)-f(x)=g(x)$ which are called {\em difference equations}. A {\em valued difference field} is a valued field equipped with a distinguished automorphism which fixes the valuation ring setwise. These structures have attracted attention after their relevance in~\cite{HFrob} as a tool for analyzing difference varieties. 

The results of Ax\&Kochen and Ershov, see \cite{koch}, on valued fields have been very influential on the algebraic theory of valued fields and lead to an asymptotic solution of Artin's conjecture that over the $p$-adics every homogeneous polynomial of degree $d$ with more than $d^2$ variables has a nontrivial zero\footnote{The conjecture is indeed false.}. Hence it is very natural to ask whether one can obtain analogues of these results in the context of valued difference fields. Most notably, in \cite{BMS}, Belair, Macintyre and Scanlon provide axiomatization and relative quantifier elimination for the field of Witt vectors over the algebraic closure of $\mathbb{F}_p$ equipped with the lifting of the Frobenius automorphism. One of the key ingredients of these results is the study of zeroes of difference polynomials in one variable. This is purely algebraic in nature and may be seen as a first step towards the study of difference varieties in terms of valuation theory. The success of \cite{BMS} inspired further research and more results in the same direction were proved in \cite{ad}, \cite{vdca} and \cite{Pal}. It is worth mentioning that these results were used in \cite{HilsCher} to show that valued difference fields are concrete examples of NTP2 theories. The main results of this paper are again axiomatization and relative quantifier elimination for valued difference fields, see Section~\ref{cqe1}, but in a much more general context than the results mentioned above. In Section~\ref{trans} we apply these results to the transseries field considered as a valued difference field with the automorphism that sends each element $f(x)$ to $f(x+1)$. 

Let $K$ be a valued difference field with valuation $v$, valuation ring $\ca{O}$ whose maximal ideal is $\fr{m}$, value group $\Gamma$ and distinguished automorphism $\sigma$. Then $\sigma$ induces an automorphism of the residue field $\bk:=\ca{O}/\fr{m}$, denoted $\bar{\sigma}$. Likewise $\sigma$ induces an order-preserving automorphism of the value group, which we also
denote by $\sigma$:
$$\gamma \mapsto \sigma(\gamma):= v(\sigma(a)) \text{ where } v(a)=\gamma.$$ Up to this point model theoretic results on valued difference fields are obtained in restricted settings determined by the action of $\sigma$ on the value group. For example in \cite{BMS}, $\sigma$ is assumed to induce the identity map on $\Gamma$. In general the action of $\sigma$ on $\Gamma$ can be quite complicated, see Hahn difference fields in Section~\ref{prelim} and also Section~\ref{trans}. Our key improvement in this paper is omitting all assumptions about the action of $\sigma$ on the value group except what is already implied by requiring that it fixes the valuation ring set-wise. This is achieved via generalizing and unifying various techniques from the aforementioned results. The notion of {\em regularity} from the Ph. D. thesis of the second author~\cite{Gonenctez} (also used in \cite{Pal} independently) turns out to be a very efficient tool for understanding the interaction between pseudo-Cauchy sequences and $\sigma$-polynomials, see Section~\ref{pcandsigma}. 

For the main content of this paper we need the results of Section~\ref{pcandsigma} only for $\sigma$-polynomials in one variable but we also present straightforward generalizations to the multivariable setting at the end of Section~\ref{pcandsigma}. These results are very much in the spirit of tropical geometry and we hope that they will find applications in the study of difference varieties. Indeed, we provide a straightforward proof of Kapranov's Theorem as a consequence, see Theorem~\ref{kapranov}. A tropical approach to difference geometry with the tools introduced here also seems within reach but this is a seperate goal to be pursued elaborately in further research.

Henselianity plays a crucial role in~\cite{koch} and all similar results about the elementary theory of valued fields. There are various attempts for a suitable definition of {\em $\sigma$-henselianity} to fullfil that role in valued difference fields. The one we use (see Definition~\ref{shensel2}) was introduced in~\cite{vdca}. It was initially intended to deal with {\em contractive} valued difference fields\footnote{That is; valued difference fields where $\sigma(\gamma)>n\gamma$ for all $n$ whenever $\gamma>0$.} but in Section~\ref{newtonhensel1} we show that this notion $\sigma$-henselianity is suitable for the general setting as well. It is worth noting that $\sigma$-henselianity
implies that the residue difference field is {\em linear difference closed}. 
That is,
if $K$ is a
$\sigma$-henselian valued difference field as in Definition~\ref{shensel2}, then
for all $\alpha_0, \dots, \alpha_n \in \bk$ with 
$\alpha_i \neq 0$ for some $i$, the equation 
$$1+\alpha_0x+\alpha_1\bar{\sigma}(x)+\cdots+\alpha_n\bar{\sigma}^n(x)=0$$
has a solution in $\bk$; see Lemma~\ref{hensax1}. In particular our results in Section~\ref{cqe1} do not apply when
when $\bar{\sigma}$ is the 
identity on the residue field. Let us note that this restriction is present in each of ~\cite{BMS},~\cite{vdca},~\cite{Pal}.

In Section~\ref{trans} we apply our results to transseries. Transseries, also referred to as LE-series, present themselves in real differential algebra as real closed differential fields, see~\cite{vdHLN}. Morevoer there is a lot of additional structure on transseries; exponentiation, valuation, composition, integration. We will be considering a transseries field, $\mathbb T$, as a valued field equipped with a {\em right composition} map. If $g(x)$ is a positive, infinite transseries then the map given by $\f(x) \mapsto f(g(x))$ is an automorphism of $\mathbb T$ and it fixes the valuation ring setwise. As such $\mathbb T$ becomes a valued difference field.  One natural choice for $g(x)$ as above is $g(x)=x+1$. Let us note that in this case the automorphism is neither value preserving, nor contractive, nor multiplicative; the cases considered in ~\cite{BMS}, ~\cite{vdca}, ~\cite{Pal} respectively. 
Unfortunately, the residue field of the natural valuation on $\mathbb T$ is real closed. In particular, the residue difference field is not linear difference closed. However, when $g(x)=x+1$, we can pass to a coarsening of the natural valuation whose residue difference field is linear difference closed and then apply the results of Section~\ref{cqe1}. This trick can not be applied for an arbitrary $g(x)$. In the particularly interesting case $g(x)=e^x$, the linear difference equation $$f(e^x)-f(x)=1$$ does not have a solution in $\mathbb T$ and worse the reduced equation does not have a solution in the residue difference field of any coarsening. 

We would like to thank Lou van den Dries and Joris van der Hoeven for the help they provided in issues concerning transseries. Part of this research was conducted at Nesin Mathematics Village, we are thankful for their hospitality. 
\end{section}

\begin{section}{Preliminaries}\label{prelim}

\noindent
Throughout, $\mathbb{N}=\{0,1,2,\dots\}$, and $m,n$ range over $\mathbb{N}$. 
We let $K^\times=K\setminus \{0\}$ be the multiplicative group of a field $K$. 
Some of the basic concepts we use are introduced below. For the rest (difference fields, $\sigma$-polynomials and their Taylor expansions, valued fields, etc.) we follow the notations and conventions given in the preliminaries section of \cite{ad}.

\bigskip\noindent
{\bf Ordered difference groups.} An {\em ordered difference group} is 
an ordered abelian group equipped with a distinguished order-preserving
automorphism. Note that an order preserving automorphism is a strictly increasing function. We consider an ordered difference group in the obvious 
way as a structure for the language 
$\{ 0, +, -, <, \sigma \}$, 
where the unary function symbol $\sigma$ is interpreted as the distinguished automorphism. Let $\Delta \subseteq \Gamma$ be an extension of 
ordered difference groups, and $\gamma \in \Gamma$.
We define $\Delta \langle \gamma \rangle$ to be the smallest ordered difference subgroup of $\Gamma$ containing $\Delta$ and $\gamma$. 
For $\i=(i_0, \dots, i_n) \in \mathbb{Z}^{n+1}$
we put $$\bsigma^{\i}(\gamma):= \sum_{k=0}^n
i_k \sigma^k(\gamma).$$

\medskip\noindent
Consider the polynomial ring $\mathbb{Z}[\sigma]$ where $\sigma$ is taken as 
an indeterminate.
We construe the ordered difference group $\Gamma$ as a
$\mathbb{Z}[\sigma]$-module as follows: for 
$$\tau=\sum_{k=0}^{n}i_k\sigma^k \in \mathbb{Z}[\sigma], \quad \gamma\in \Gamma,$$
we set $\tau(\gamma):=\bsigma^{\i}(\gamma)$ where $\i=(i_0, \dots, i_n) \in \mathbb{Z}^{n+1}$. We also consider each ordered difference subgroup of 
$\Gamma$ as a $\mathbb{Z}[\sigma]$-submodule of $\Gamma$.
Let $\gamma \in \Gamma \setminus \Delta$. We define the {\em annihilator of 
$\gamma$ modulo $\Delta$} to be
$$\ann_{\Gamma / \Delta}(\gamma):=\{\tau \in
\mathbb{Z}[\sigma]: \tau(\gamma) \in \Delta\},$$ 
which is an ideal of $\mathbb{Z}[\sigma]$. 

{\bf Valued difference fields.} A {\em valued difference field\/} is a valued field $\ca{K}=(K, \Gamma, \bk; v,\pi)$ 
where $K$ is not just a field, but a difference field whose
difference operator $\sigma$ satisfies
$\sigma(\ca{O}_v)= \ca{O}_v$. It follows that
$\sigma$ induces an automorphism of the residue field:
$$ \pi(a) \mapsto \pi(\sigma(a)):\ \bk \to \bk, \quad a \in \ca{O}.$$  
We denote this automorphism by $\bar{\sigma}$, and $\bk$ equipped with
$\bar{\sigma}$ is called the {\em residue difference field of $\ca{K}$}. Likewise, $\sigma$ induces an order preserving automorphism of the value group $\Gamma$; which we also denote by $\sigma$. 
Furthermore
$$v(\bsigma(y)^{\i})=\bsigma^{\i}(\gamma)$$ for all $y\in K^{\times}$ 
with $v(y)=\gamma$ and hence $v:K^{\times} \to \Gamma$ is a 
morphism of $\mathbb{Z}[\sigma]$-modules.

\medskip\noindent
Let $\ca{K}$ be a valued difference field. The difference operator $\sigma$ of $K$ is also referred to as the {\em difference operator of $\ca{K}$}. By an {\em extension\/} of $\ca{K}$ we shall mean a valued difference field $\ca{K}'=(K',\dots)$ that extends $\ca{K}$ as a valued field and whose difference operator extends the difference operator of $\ca{K}'$. In this situation we also say that $\ca{K}$ is a {\em valued difference subfield of\/} $\ca{K}'$, and we indicate this by $\ca{K} \le \ca{K}'$. Such an extension is called {\em immediate\/} 
if it is immediate as an extension of valued fields.
In dealing with a valued difference field $\ca{K}$ as above $v$ also denotes the valuation of any extension of $\ca{K}$ that 
gets mentioned (unless specified otherwise), and any difference 
subfield $E$ of $K$ is construed 
as a valued difference subfield of $\ca{K}$ in the obvious way. Whenever we say $\ca{K}$ is a valued difference field, it should be understood that $\ca{K}=(K, \Gamma, \bk; v, \pi)$; thus fixing the notations for the underlying field, value group, residue field and valuation. Likewise $\ca{K}'$ will always be $(K', \Gamma', \bk'; v',\pi')$. In case $\ca{K} \leq \ca{K}'$ we write $v$, $\sigma$, $\pi$ for $v'$, $\sigma'$ and $\pi'$ respectively.


Let $\ca{K}^h= (K^h, \Gamma, \bk;\dots)$ be the henselization of the 
underlying valued field of $\ca{K}$. By the universal property of ``henselization''
the operator $\sigma$ extends uniquely to an automorphism $\sigma^h$ of the field $K^h$ such that $\ca{K}^h$ with $\sigma^h$ is a valued difference field. Accordingly we shall view $\ca{K}^h$ as a valued difference field, making it thereby an immediate extension of the valued difference field $\ca{K}$.

Given an extension  $\ca{K} \leq \ca{K}'$ of valued difference fields and 
$a \in K'$, we define $\ca{K} \langle a \rangle$ to be the 
smallest valued difference subfield of $\ca{K}'$ extending $\ca{K}$ and 
containing $a$ in its underlying difference field; thus
the underlying difference field of $\ca{K} \langle a \rangle$ 
is $K \langle a \rangle$.

\bigskip\noindent
{\bf Hahn difference fields.} 
Let $\bk$ be a field and $\Gamma$ an ordered abelian group. This gives the Hahn field
$\bk((t^{\Gamma}))$ whose elements are the formal sums
$a=\sum_{\gamma \in \Gamma} a_\gamma t^{\gamma}$  
with $a_\gamma \in \bk$ for all $\gamma$, with well-ordered {\em support\/}
$\{\gamma:\ a_\gamma \neq 0\} \subseteq \Gamma$. With $a$ as above, we define the 
valuation $v: \bk((t^{\Gamma}))^{\times} \to \Gamma$ by
$v(a):=\min \{\gamma: a_\gamma \neq 0\}$,
and the surjective ring morphism  
$\pi:\ \ca{O}_v \to \bk$ by $\pi(a):=a_0$.
In this way we obtain the (maximal) valued field
$\ca{K}=(\bk((t^\Gamma)), \Gamma, \bk;v,\pi)$ to which we also just refer to
as the {\em Hahn field} $\bk((t^\Gamma))$.

\medskip\noindent
Suppose that the field $\bk$ is equipped with an
automorphism $\bar{\sigma}$ and that the ordered group $\Gamma$ is equipped with an 
order-preserving automorphism $\sigma$. Then
$$\sum_{\gamma} a_\gamma t^{\gamma} \mapsto 
\sum_{\gamma} \bar{\sigma}(a_\gamma) t^{\sigma(\gamma)}$$
is an automorphism, to be denoted by $\sigma$, 
of the field $\bk((t^\Gamma))$, with $\sigma(\ca{O}_v)=\ca{O}_v$.
We consider the three-sorted structure 
$(\bk((t^\Gamma)), \Gamma, \bk;\ v,\pi)$, with the field $\bk((t^\Gamma))$
equipped with the
automorphism $\sigma$ as above, as a valued difference field,
and also refer to it as the {\em Hahn difference field} 
$\bk((t^\Gamma))$.

\bigskip\noindent
{\bf Pseudo-Cauchy Sequences.} 
A {\em well-indexed sequence}
is a sequence $\{a_\rho\}$ indexed by the elements $\rho$ of some
nonempty well-ordered set without largest element; and throughout 
``eventually'' means ``for all sufficiently large $\rho$'' in the context of a well-indexed sequence.
Let $\ca{K}$ be a valued field and $\{a_\rho\}$ a well-indexed sequence from $K$.
\begin{definition}\label{pcs} The sequence $\{a_\rho\}$ is a pseudo-Cauchy sequence (pc-sequence)
if for some index
$\rho_0 $, $$\rho'' > \rho' > \rho \ge \rho_0\ \Longrightarrow\ 
v(a_{\rho''}-a_{\rho'}) > v(a_{\rho'}-a_\rho).$$
For $a$ in some valued field extension of $\ca{K}$, we say that $a$ is a pseudo-limit of $\{a_\rho\}$ if 
$\{v(a-a_\rho)\}$ is eventually increasing, denoted $a_\rho \leadsto a$. 
\end{definition}

For $\rho_0$ 
as above, and put 
$$\gamma_\rho :=v(a_{\rho'}-a_\rho)$$ for
$\rho'>\rho \ge \rho_0$; this depends only on $\rho$. 
Then $\{\gamma_\rho\}_{\rho \ge \rho_0}$ is strictly increasing.
The {\em width} of $\{a_\rho\}$ is the set
$$\{\gamma \in
\Gamma\cup\{\infty\}:\gamma>\gamma_\rho\ \mbox{for all}\ \rho \ge \rho_0\}.$$ 
Its significance is that if $a,b\in K$ and $a_\rho \leadsto a$, then 
$a_\rho \leadsto b$ if and only if $v(a-b)$ is in the width of $\{a_\rho\}$.
A useful observation about pc-sequences is that if $\{a_\rho\}$ is a pc-sequence 
in an expansion of a valued field (for example, in a valued difference field), 
then $\{a_\rho\}$ has a pseudolimit in an elementary extension of that
expansion.

\end{section}

\begin{section}{Regularity and Pseudoconvergence}\label{pcandsigma}
Let $\ca{K}$ be a valued difference field. For a $\sigma$-polynomial
$F(x)=\sum_{\i}a_{\i}\bsigma(x)^{\i}$
over $K$ and $\gamma \in \Gamma$ define 
$F_v(\gamma) \colon=\min_{\i}\{v(a_{\i})+\bsigma^{\i}(\gamma)\}$.
Thus $F$ induces a map
\begin{align*}
F_v\colon &\Gamma \to \Gamma \\
       & \gamma \mapsto F_v(\gamma)
\end{align*}
which is strictly increasing whenever $F(x)$ is nonzero and the constant term of $F$ is equal to zero (since $\sigma$ is an order preserving automorphism of $\Gamma$). In general, $F_v$ is strictly increasing on an initial segment of $\Gamma$ and constant on the complement of this initial segment. 
Note that if $F$ is a $\sigma$-monomial then for all $a \in K$ we have $vF(a)=F_v(\gamma)$, where $v(a)=\gamma$. This obviously does not hold for $\sigma$-polynomials.

\begin{definition}
Let $F(x)$ be a $\sigma$-polynomial over $K$. An element $a \in K$ is regular for $F$ if   $$v(F(a))=F_v(\gamma),$$ 
where $\gamma=v(a)$, otherwise $a$ is irregular for $F$. We say that  $a \in K$ is regular over a subfield $E$ of $K$ if $a$ is regular for all $\sigma$-polynomials with coefficients from $E$.
\end{definition}

\noindent Every element $a \in K$ is regular for all  $\sigma$-monomials over $K$ and $0$  is regular for  $F(x)$  if and only if $F(0)=0$.
Note that $a$ is irregular for $F(x)$ if and only if $vF(a)>F_v(\gamma)$. 


\begin{lemma}\label{reg2} Let $a, b \in K$ be such that $v(b-a)>v(a)$ and $F(x)$ a $\sigma$-polynomial over $K$. Then $a$ is regular for $F(x)$ if and only if $b$ is regular for $F(x)$. 
\end{lemma}
\begin{proof} Suppose that $a$ is regular for $F(x)=\sum_\i a_\i \bsigma(x)^\i$ and $v(a)=\gamma$. Pick $\j$ with 
$v(a_\j)+\bsigma^\j(\gamma)=F_v(\gamma)=vF(a)$ and let 
$$G(x)=\frac{F(ax)}{a_\j\bsigma(a)^\j}=\sum_\i \frac{a_\i \bsigma(a)^\i }{a_\j\bsigma(a)^\j}\bsigma(x)^\i.$$ Then $G(x)$ has coefficients from the valuation ring and since $a$ is regular for $F(x)$, $vG(1)=0$. Therefore we obtain $\overline{G}(1) \neq 0 \in \bk$ where $\overline{G}(x)$ is the reduced $\bar{\sigma}$-polynomial over $\bk$.  Since $v(b-a)>v(a)$, there is $c$ in the valuation ring with $b=ac$ and $\bar{c}=1$. Now, $\overline{G}(\bar{c}) \neq 0$ and hence $vG(c)=0$. This leads to
$$vF(b)=v(a_\j)+\bsigma^\j(\gamma)=F_v(\gamma)$$
and so $b$ is also regular for $F(x)$. 
\end{proof}

If $\{a_\rho\}$ is a pc-sequence in a valued field $K$ and
$a_\rho \leadsto a$ with $a\in K$, then for an ordinary nonconstant 
polynomial $f(x) \in K[x]$ we have $f(a_\rho) \leadsto f(a)$, see \cite{kaplansky}.
For certain valued difference fields, as in ~\cite{vdca}, the same is true for nonconstant $\sigma$-polynomials but that is not the case in general.  
In the context of a value-preserving automorphism, this issue was resolved in 
\cite{BMS} via the notion of {\em equivalent pc-sequences\/} under the assumption that the automorphism is not of finite order over the residue field. We will also need this assumption, which we explicitly state as an axiom. 

\vglue.3cm
\noindent
{\bf Axiom 1.}\ For each integer $d>0$ there is
$y\in \bk$ with $\bar{\sigma}^d(y)\ne y$. 

\medskip\noindent Also, the focus of the current paper is valued difference fields with residue characteristic zero, and so we assume this throughout the rest even though some results are valid without this assumption.

\begin{definition}\label{equiv.pc} Two pc-sequences
$\{a_\rho\},\{b_\rho\}$
in a valued field are equivalent if for all $a$ in all 
valued field extensions,
$a_\rho \leadsto a
\Leftrightarrow b_\rho \leadsto a.$
\end{definition}

\noindent
This is an equivalence relation on the set of pc-sequences with 
given index set and
in a given valued field, and: 

\begin{lemma}
 Two pc-sequences $\{a_\rho\}$ and $\{b_\rho\}$ in a valued field 
are equivalent if and only if they
 have the same width and a
 common pseudolimit in some valued field extension.
\end {lemma}

We shall prove that given a pc-sequence $\{a_\rho\}$ from a valued difference field $\ca{K}$ which satisfies Axiom $1$, and a $\sigma$-polynomial $F(x)$ over $\ca{K}$ there is an equivalent pc-sequence $\{b_\rho\}$ such that $\{F(b_\rho)\}$ is also a pc-sequence. Construction of $b_\rho$, even when assuming specific behaviour of $\sigma$ as in \cite{BMS}, is quite complicated. Appropriately using regular elements, this can be achieved in straightforward manner. First we state a well-known fact which will allow us to use Axiom 1 effectively. 

\begin{lemma}\label{zariskitop} Let $k \subseteq k'$ be a field extension, and
$g(x_0,\dots,x_n)$ a nonzero polynomial over $k'$. Then there is
a nonzero polynomial $f(x_0,\dots,x_n)$ over $k$ such that whenever
$y_0,\dots,y_n\in k$ and
$f(y_0,\dots,y_n)\ne 0$, then 
$g(y_0,\dots,y_n) \ne 0$.
\end{lemma}

The existence of regular elements as described below will be the key tool for constructing an equivalent pc-sequence as in the above discussion. 

\begin{lemma}\label{regforpc}
Suppose that $\ca{K} \leq \ca{K}'$ is an extension of valued fields, and $\ca{K}$ satisfies Axiom $1$.  Let $F$ be a $\sigma$-polynomial over $K'$, and $\gamma \in \Gamma$. Then there exists $a \in K$ such that $v(a)=\gamma$ and $a$ is regular for $F$.
\end{lemma}
\begin{proof}
Let $F=\sum a_\i \sigma^\i(x)$ and take $b \in K$ with $v(b)=\gamma$. Set 
$$G(x):=\frac{F(bx)}{a_\j \bsigma(b)^\j}=\sum_{\i} \frac{a_\i\bsigma(b)^\i}{a_\j\bsigma(b)^\j}\bsigma(x)^{\i}$$
where $\j$ is such that $F_v(\gamma)=v(a_\j)+\bsigma^\j(\gamma)$. 
So the coefficients of $G(X)$ are in the valuation ring of $K'$, with one coefficient equal to $1$. So the the reduced $\bar{\sigma}$-polynomial $\overline{G}(x)$ over $\bk'$ is nonzero. 

By Lemma~\ref{zariskitop}, there is a nonzero $\bar{\sigma}$-polynomial $h(x)$ over $\bk$ such that for all $\alpha \in \bk$, $\overline{G}(\alpha) \neq 0$ whenever $h(\alpha)\neq 0$. Since $\bar{\sigma}$ is not of finite order as an automorphism of $\bk$, there exists a nonzero element $\alpha \in \bk$ such that 
 $h(\alpha) \neq 0$ (see~\cite{cohn}, p. 201) and hence $\overline{G}(\alpha) \neq 0$. So we can take $c \in K$ with $v(c)=0$ and $\bar{c}=\alpha$. Then 
$v(G(c))= 0$, $v(bc)=v(b)=\gamma$ and 
$$v(F(bc))=v(a_{\j}\sigma(b)^{\j})=F_v(\gamma).$$ 
That is; $a=bc \in K$ is  regular for $F$.
\end{proof}

\begin{remark}\label{regremark} It is clear from the proof of the above lemma that, under the same hypothesis, given a finite set $\Sigma$ of $\sigma$-polynomials over $\ca{K}'$ and $\gamma \in \Gamma$ we can find $a\in K$ of valuation $\gamma$ which is regular for every $\sigma$-polynomial in $\Sigma$.
\end{remark}

\begin{theorem}\label{adjustment1}
Suppose $\ca{K}$ satisfies Axiom $1$,
$\{a_\rho\}$ is a pc-sequence from $K$ and $a_\rho \leadsto a$ in an
extension.  Let $\Sigma$ be a finite set of $\sigma$-polynomials $F(x)$
over $K$. Then there is a pc-sequence $\{b_\rho\}$ from $K$,
equivalent to $\{a_\rho\},$ such that $F(b_\rho)
\leadsto F(a)$ for each nonconstant $F$ in $\Sigma.$
\end{theorem}
\begin{proof} First, let us assume that $\Sigma$ consists of a single nonconstant $\sigma$-polynomial
$F(x)$ over $K$.  Put $\gamma_\rho:=v(a_\rho-a) \in \Gamma$; then $\{\gamma_\rho\}$ is an eventually
strictly increasing sequence. Also set 
$$G(x):=F(a+x)-F(a)=\sum_{|\i| \geq 1}F_{\i}(a)\bsigma(x)^{\i}.$$
Note that $G(x)$ is a nonzero $\sigma$-polynomial which has constant term zero and  its coefficients are in $K\langle a\rangle$.  

For all $\rho$ we choose, using Lemma~\ref{regforpc}, $c_\rho \in K$ of valuation $\gamma_\rho$ such that $c_\rho$ is regular for $G(x)$.  Now, define $b_\rho:=a_{\rho+1}+c_\rho$. Then $v(a-b_\rho)=v(a-a_{\rho+1}-c_\rho)=\gamma_\rho$ eventually, so $b_\rho \leadsto a$. Moreover,
$$v(b_{\rho+1}-b_\rho)=v(a_{\rho+2} + c_{\rho+1} - a_{\rho+1} - c_\rho)=\gamma_\rho$$
eventually and hence $\{b_\rho\}$ is equivalent to $\{a_\rho\}$. Since $c_\rho$ is regular for $G(x)$ and $$v(b_\rho-a-c_\rho)=v(a_{\rho+1}+c_\rho-a-c_\rho)=\gamma_{\rho+1}>v(c_\rho)$$ 
eventually, by Lemma~\ref{reg2}, $b_\rho-a$ is regular for $G(x)$ eventually. Then
$$v\big(F(b_\rho)-F(a)\big)=vG(b_\rho-a)=G_v(\gamma_\rho)$$
eventually. Since $G(x)$ has constant term equal to zero, $G_v$ is a strictly increasing function and hence $F(b_\rho) \leadsto F(a)$. The general case can be obtained using Remark~\ref{regremark}. 
\end{proof}

\begin{corollary}\label{corpc}
The same result, where $a$ is removed and one only asks that 
$\{G(b_\rho)\}$ is a pc-sequence. 
\end{corollary}
\begin{proof}
Put in an $a$ from an elementary extension.
\end{proof}

\noindent
Using the above results it is easy to obtain the following theorem which will be crucial at later stages. See \cite{BMS} for a detailed treatment. 
 
\begin{theorem} \label{crucial.result.nonwitt}
Let $\ca{K}$ be a valued difference field satisfying Axiom $1$. Let
$\{a_\rho\}$ be a pc-sequence from $K$ and let $a$ in some extension of $\ca{K}$ be such that $a_\rho \leadsto a$.
Let $G(x)$ be a $\sigma$-polynomial over $K$ such that
\begin{enumerate}
\item[(i)]
$G(a_\rho) \leadsto 0$, 
\item[(ii)] $ G_{(\l)}(b_\rho) \not\leadsto 0$ whenever $|\l|\geq 1$ and
$\{b_\rho\}$ is a pc-sequence in $K$ equivalent to $\{a_\rho\}$.
\end{enumerate}
Let $\Sigma$ be a finite set of $\sigma$-polynomials $H(x)$ over $K$. 
Then there is a
pc-sequence $\{b_\rho\}$ in $K$, equivalent to $\{a_\rho\}$, such that 
$G(b_\rho) \leadsto 0$, and
$H(b_\rho) \leadsto H(a)$ for every nonconstant $H$ in $\Sigma$.
\end{theorem}

\medskip\noindent
{\bf The Multivariable Case and Kapranov's Theorem:} For our model theoretic treatment of valued difference fields we only need to deal with $\sigma$-polynomials in one variable. However, we will provide generalizations of some of the results of this section to the multivariable case with the hope that they can be useful for other applications. A particular goal would be to establish a basis for tropical difference geometry.

Let $F(x_1, x_2, \dots, x_n)=\sum_{\i} M_{\i}$ be a multivariable $\sigma$-polynomial,
where $M_{\i}$'s are its monomials. Then each monomial $M_{\i}$ induces  a function 
\begin{align*}
&{M_{\i}}_v:\Gamma^n \to \Gamma \\ &(\gamma_1,\dots,\gamma_n) \mapsto v(M_{\i}(x_1,\dots,x_n))
\end{align*} 
where $v(x_i)=\gamma_i$. We set 
$F_v(\gamma_1,\dots,\gamma_n):=\min_{\i}\{{M_{\i}}_v (\gamma_1,\dots,\gamma_n)\}$. Note that when $F$ is an ordinary multivariable polynomial, $F_v$ is nothing but the {\em tropicalization} of $F$.  
We define a {\em regular tuple} for $F$ to be an n-tuple $a \in K^n$ such that  
$F_v(v(a))=v(F(a))$.

We see $\Gamma^n$ as a group with component-wise addition and equipped with the natural {\em partial} ordering obtained from the ordering on $\Gamma$. We say that $\gamma, \theta \in \Gamma^n$ are {\em comparable} if $\gamma \leq \theta$ or $\theta \leq \gamma$. We also extend all structure on $K$ to $K^n$ componentwise and we use the same notations that we use for elements of $\Gamma$ and $K$ for n-tuples. So for $a=(a_1,\dots a_n) ,b=(b_1,\dots,b_n) \in K^n$ we have:
\begin{align*}
v(a)&=(v(a_1),\dots,v(a_n));\\
v(a)&<v(b) \Leftrightarrow v(a_i)<v(b_i) \text{ for } i=1,\dots,n; \\
ab&=(a_1b_1,\dots, a_nb_n);\\
a+b&=(a_1+b_1,\dots,a_n+b_n).
\end{align*}
It is easy to see that for $a,b \in K^n$, we have $v(ab)=v(a) + v(b)$ and 
$$v(a+b)\leq \min \{v(a),v(b)\}$$
whenever $v(a)$ and $v(b)$ are comparable. In particular, if $v(a)<v(b)$ then $v(a+b)=v(a)$.
A pc-sequence from $K^n$ is a sequence such that each of its coordinates is a pc-sequence in $K$ and it pseudo-convergences   to an n-tuple if each of its coordinates pseudo-converge to the corresponding coordinate of that n-tuple. Two pc-sequences from $K^n$ are equivalent if the corresponding pc-sequences from each coordinate are equivalent. 

Then it is straightforward to check that all the previous results of this section are valid in the multivariable context. In particular, we note following consequence of the proof of the Lemma~\ref{regforpc}:

\begin{lemma}\label{correg}
Let $a\neq 0 \in K^n$ and $F(x)$ be a nonzero multivariable $\sigma$-polinomial. Then $a$ is regular for $F(x)$ if and only if $(1,\dots,1)$ is not a zero of reduced polynomial $\overline{F(ax)/d}$ over $\k$ for some $d \in K^n$ of valuation
$F_v(v(a))$. More generally, for all $b \in K^n$ such that $v(b)=(0,\dots,0)$, the element $ab$ is irregular for $F$ if and only if $ab$ is a root of $F$ or  $(\bar{b}_1,\dots,\bar{b}_n)$ is a root of $\overline{F(ax)/d}$. 
\end{lemma}


Now let us explain how these tools relate to tropical geometry. Let $F(x)=\sum_{\i} M_{\i}$ be multivariable 
$\sigma$-polinomial, with monomials $M_{\i}$, over a non trivially valued field $\ca{K}=(K,\Gamma, 
\bk;v)$. A {\em tropical zero} of $F$ is an element 
$\gamma \in \Gamma^n$, such that $F_v(\gamma)$ is equal to ${M_{\i}}_v(\gamma)={M_{\j}}_v(\gamma)$ for some $\i \neq \j$. Note that if $F$ is an ordinary polynomial in one variable then it has finitely many tropical zeroes. The same assertion is false for multivariable polynomials and $\sigma$-polynomials even in one variable (e.g. if $\sigma$ is the identity on $\Gamma$ then $F(x)=\sigma(x) - x$ has infinitely many tropical zeroes). This difference is actually the reason why pseudo-convergence is preserved under ordinary polynomials in one variable but not under $\sigma$-polynomials. So it is no surprise that the tools we introduced to deal pseudo-convergence are actually closely related to tropical geometry.


 
\medskip\noindent
If $a \in K^n$ is a zero of $F$ then $v(a)$ is a tropical zero of $F$. One of the essential results in tropical geometry is Kapranov's Theorem (see Theorem 2.1.1. \cite{Kap}) which asserts the converse when $K$ is algebraically closed; namely that if $\gamma$ is a tropical zero of $F$ then there is $a \in K^n$ with $v(a)=\gamma$ and $F(a)=0$. Using the lemma below we provide an alternative proof of this fact.

\begin{lemma}\label{F_v=G_v}
Let $F$ be an ordinary one variable polynomial over non trivially valued field $(K,\Gamma, v)$, 
$\gamma \in \Gamma$ be a tropical zero of $F$ and $b \in K$ of valuation $\gamma$. Consider the 
polynomial
$G(x):=\sum_{i \geq 1} F_{i}(b)x^i$. Then $G_v(\gamma)=F_v(\gamma)$ 
\end{lemma}

\begin{proof}
Since it enough to show that $F_v(\gamma)$ and $G_v(\gamma)$ agree over an extension of $(K,v)$ we can suppose that $K$ is algebraically closed. Let $x$ be of valuation $\gamma$.  Since the residue field is infinite, by the multivariable analogue of Lemma~\ref{regforpc} we can choose a regular $\epsilon \in K^n$ of valuation $\gamma$ which is regular for $G$.
Then $v(\epsilon+b)\geq 
\gamma$ and $F(\epsilon+b)-F(b)=G(\epsilon)$. Since 
$$F(\epsilon+b)\geq F_v(\gamma) \quad \& \quad F(b)\geq F_v(\gamma)$$ we have 
$v(G(\epsilon))\geq F_v(\gamma)$.

The reduced polynomial $\overline{F(bx)/d}$ with $v(d)=F_v(\gamma)$, as in Corollary~\ref{correg}, is nonzero and since $\bk$ is algebraically closed we can pick a root of it which provides an irregular element for $F$ of the form $by$ with $v(y)=0$. On the other hand,  as above, there exists regular elements for $F$ for an arbitrary value in $\Gamma$. Now if $b$ is irregular for $F$ then choose $c \in K$ of valuation $\gamma$ regular for $F$, if $b$ is regular for $F$ then choose $c$ irregular for $F$ of valuation $\gamma$. Then $v(c-b)=\gamma$ by \ref{correg} and $v(F(c)-F(b))=\min\{F(c),F(b)\}=F_v(\gamma)$. Hence we have 
$F_v(\gamma)=v(G(c-b))\geq G_v(\gamma)$. 
\end{proof}
\begin{theorem}\label{kapranov}[Kapranov's theorem]
If $F(x)$ is a nonzero multivariable polynomial  over an algebraically closed valued field $(K,\Gamma,v)$ and $\gamma \in \Gamma^n$ is tropical zero of $F$ then there is root of $F$ of valuation  $\gamma$. 
\end{theorem} 
\begin{proof} Let $a:=(a_1,\dots, a_{n-1},a_n) \in K^n$ be of valuation $\gamma$. Since $\gamma$ is a tropical zero of $F$, we can w.l.o.g. suppose that $a$ is   irregular for $F(x)$: In fact, if not; repeating the proof above,  by \ref{correg} we can consider the polynomial  $\overline{F(ax)/d}$ over $\bk$; since $\bk$ is algebraically closed we can take a root of it which provides an irregular element for $F(x)$ of the form $ab$ with $v(b)=(0,\dots,0)$. Remark in this case that by \ref{correg} $a_n$ is irregular for the one variable polynomial  $F^a:=F(a_1,\dots, a_{n-1}, y)$: the fact that $(1,\dots,1)$ is a root of $\overline{F(ax)/d}$ implies $1$ is a root of $\overline{F^a(a_ny)/d}$. Now it is enough to find a zero of $F^a$ of valuation  $v(a_n)$. 

Set $b_0:=a_n$ and $\theta:=v(b_0)$. We will first find $b_1 \in K$ of valuation $\theta$ such that $v(b_1-b_0)>\theta$ and $F^a(b_1)>F^a(b_0)$. 
Set  $$G(x):= \sum_{i \geq 1} F^a_{(i)}(b_0)x^i.$$ By divisibility of the value group $\Gamma$, we can pick $\delta$ such that $G_v(\delta)=v(F^a(b_0))$. 
By Lemma~\ref{F_v=G_v}, we have $G_v(\theta)=F^a_v(\theta)$ and since $b_0$ is irregular for $F^a$ we get $v(F^a(b_0))>F^a_v(\theta)$. Putting these together with the fact that $G_v$ is strictly increasing we conclude that $\delta > \theta$.  Let $\epsilon$ be such that $v(\epsilon)=\delta$. Set $b_1=b_0 + \epsilon u$ where $u$ is a new variable. 
Then $G(\epsilon u)+ F^a(b_0)= F^a(b_1)$. Since 
$G_v(\delta)=v(F^a(b_0))$ the polynomial $$H(u):=\frac{G(\epsilon u)}{F^a(b_0)}$$ has coefficients from the valuation ring of $K$, and  the reduced polynomial $\overline{H}(u)$ over the residue field is nonzero and has constant term zero. Now, in order to have $F^a(b_1)>F^a(b_0)$, we pick a $c \in K$ of valuation zero such that $\bar{c} \in \bk$ is zero of the polynomial $\overline{H} + 1$ and set $b_1=b_0 + \epsilon c$. 

Proceeding inductively we can  construct a pc-sequence $\{b_k\}_{k \in \mathbb{N}}$ with 
$F^a(b_{k+1})>F^a(b_{k})$ if $F^a(b_k)\neq 0$ for all $k\in \mathbb{N}$.  Since $\{b_k\}$ is of algebraic type  $\{b_k\} \leadsto b_{\omega}$ for some $b_{\omega} \in K$. Since $\{F^a(b_k)\} \leadsto 0$ and $\{F^a(b_k)\} \leadsto F^a(b_{\omega})$ we have $F^a(b_{\omega})>F^a(b_k)$ for all $k \in \mathbb{N}$. Then again, either $F^a(b_{\omega})=0$ or we can continue extending our sequence as before. Hence we can construct
a pc-sequence of arbitrary length $\{b_\rho \}$ if during the process we do not get a zero $F^a$. We must obtain a zero of $F^a$  considering the cardinality of $K$.
\end{proof}

\end{section}

\section{Approximating zeroes of $\sigma$-polynomials}\label{newtonhensel1}

Let $\ca{K}$ be a valued difference field, of residue characteristic zero as usual.  
Let $G(x)$ be a $\sigma$-polynomial over $K$ of order $\le n$ and $a\in K$.
Let $\i$ range over tuples $(i_0,\dots,i_n)\in \mathbb{N}^{n+1}$, and 
likewise with $\j=(j_0,\dots,j_n)$, $\l=(l_0, \dots ,l_n)$. Much of this section is parallel with the corresponding sections of
\cite{BMS} and \cite{ad}. There are some points which require close inspection, in which case we present proofs in full detail eventhough they look identical to what is already known from \cite{BMS} and \cite{ad}; otherwise we just point out the necessary references.

\begin{definition} We say $(G,a)$ is in 
{\em $\sigma$-hensel configuration} if $G\notin K$ and there is 
$\i$ with $|\i|=1$ and $\gamma \in \Gamma$ such that
\begin{enumerate}
 \item[$(i)$] $v(G(a))=v(G_{(\i)}(a))+\bsigma^{\i}\gamma \leq v(G_{(\j)}(a)) +\bsigma^{\j}\gamma$ whenever $|\j|=1$,
 \item[$(ii)$] $v(G_{(\j)}(a))+\bsigma^{\j}\gamma < v(G_{(\j+\l)}(a)) +\bsigma^{\j+\l}\gamma$ whenever $\j, \l \neq 0$ and \linebreak $G_{(\j)} \neq 0$.
\end{enumerate}
\end{definition}

\medskip\noindent
If $(G,a)$ is in $\sigma$-hensel configuration, then 
$G_{(\j)}(a) \neq 0$ whenever $\j\ne 0$ and $G_{(\j)} \neq 0$,
so $G(a)\ne 0$, and therefore $\gamma$ as above satisfies 
$$v(G(a))=\min_{|\j|=1} v(G_{(\j)}(a))+\bsigma^{\j}\gamma,$$ so 
is unique, and we set $\gamma(G,a):=\gamma$. If $(G,a)$ is not in $\sigma$-hensel 
configuration, we set $\gamma(G,a):=\infty$. 

\begin{remark} Suppose $G$ is nonconstant, $G(a)\ne 0$, 
$v(G(a))>0$ and $v(G_{(\i)}(a))=0$ for all $\i\ne 0$ with 
$G_{(\i)} \neq 0$. Then
$(G,a)$ is in $\sigma$-hensel configuration with $\gamma(G,a)>0$.
\end{remark}

\bigskip
The definition of $\sigma$-hensel configuration above is identical with the corresponding definition in \cite{vdca} and \cite{BMS} . In order to obtain Lemma~\ref{newton2} we need to impose the following condition on the residue difference field.

\medskip\noindent
{\bf Axiom $2_n$.} If $\alpha_0,\dots,\alpha_n \in \bk$ are not all $0$, 
then the equation 
$$1+\alpha_0x + \dots + \alpha_m\bar{\sigma}^n(x)=0$$ has a solution in $\bk$.

We say that a difference field satisfies Axiom 2 if it satisfies Axiom $2_n$ for all n, such a difference field will be called {\em linear difference closed}. This axiom is very similar to Kaplansky's condition on residue fields in the study of valued fields of positive residue characteristic. One just replaces the Frobenius with 
$\bar{\sigma}$, see \cite{vdca} for a detailed study of the connection. It is shown in \cite{vdca} that Axiom 2 can not be avoided if one wants to obtain Theorem~\ref{unique.max.imm.ext}. Nonetheless it is conceivable that AKE-type results can be obtained without having Theorem~\ref{unique.max.imm.ext} and hence without requiring the residue difference field to be linear difference closed. Presently we avoid this discussion and assume throughout the rest of the paper that all valued difference fields under consideration satisfy Axiom 2. Note that Axiom 2 implies Axiom 1. The next lemma is identical to Lemma 4.4 from ~\cite{vdca}, and one can see that its proof remains valid in our context when the details are made explicit. 

\begin{lemma}\label{newton2} Suppose that $\ca{K}$ satisfies Axiom $2$, and 
$(G,a)$ is in $\sigma$-hensel configuration. Then there is $b \in K$
such that \begin{enumerate}
\item[$(1)$] $v(b-a) \ge \gamma(G,a), \quad v(G(b))>v(G(a))$, 
\item[$(2)$] either $G(b)=0$, or $(G,b)$ is in $\sigma$-hensel configuration.
\end{enumerate} 
For any such $b$ we have $v(b-a) = \gamma(G,a)$ and $\gamma(G,b)>\gamma(G,a)$. \end{lemma}
\begin{proof} Let $\gamma=\gamma(G,a)$, pick $\epsilon \in K$ with $v(\epsilon)=\gamma$. Let 
$b=a+\epsilon u$ where $u \in K$ is to be determined later; we only impose 
$v(u)\ge 0$ for now. 
Consider $$G(b)=G(a) + \sum\limits_{|\i| \geq 1}G_{(\i)}(a) \cdot \bsigma(b-a)^{\i}.$$ 
Therefore $G(b)=G(a) \cdot (1+ \sum\limits_{|\i| \geq 1} c_{\i} \cdot \bsigma(u)^{\i})$, where 
$$c_{\i}=\frac{G_{(\i)}(a) \cdot \bsigma(\epsilon)^{\i}}{G(a)}.$$ 
From $v(\epsilon)=\gamma$ we obtain 
$\min_{|\i|=1} v(c_{\i}) =0$ and $v(c_{\j})>0$ for $|\j|>1$. Then imposing
$v(G(b)) > v(G(a))$ forces $\bar{u}$ to be a solution of the 
equation
$$1+\sum\limits_{|\i|=1}\bar{c_{\i}} \cdot \bar{\bsigma}(x)^{\i}=0.$$ 
By Axiom $2$ we can take $u$ with this property, and then $v(u)=0$, so 
$v(b-a)=\gamma(G,a)$ and $v(G(b))>v(G(a))$. 

Assume that $G(b) \ne 0$. It remains to show that then $(G,b)$ is in
$\sigma$-hensel configuration with $\gamma(G,b)>\gamma$. 
Let $\j \neq 0$, $G_{(\j)} \neq 0$ and consider
$$G_{(\j)}(b)=G_{(\j)}(a) + \sum\limits_{\l \ne 0} G_{(\j)(\l)}(a) \cdot \bsigma(b-a)^{\l}.$$
Note that $G_{(\j)}(a) \neq 0$. 
Since $\cha{\bk}=0$, 
$v(G_{(\j)(\l)}(a))=v(G_{(\j+\l)}(a))$. 
Therefore, for all $\l\ne 0$, $$v\big(G_{(\j)(\l)}(a) \cdot \bsigma(b-a)^{\l}\big)>v(G_{(\j)}(a)),$$ 
hence $v(G_{(\j)}(b))=v(G_{(\j)}(a))$. If $|\i|=1$, then 
$\theta \mapsto \bsigma^{\i}(\theta)$ is an automorphism of $\Gamma$.
Since $G(b)\ne 0$, it follows that  
we can pick $\gamma_1 \in \Gamma$ such that 
$$G(b)=\min_{|\i|=1} v(G_{(\i)}(b)) + \bsigma^{\i}\gamma_1.$$ 
Note that $\gamma_1 > \gamma$ because $v(G(b))>v(G(a))$ and $v(G_{(\i)}(b))=v(G_{(\i)}(a))$ for $\i \neq 0$. 
Also for $\i,\j \neq 0$ and $\theta \in \Gamma$ with $\theta>0$ we have $\bsigma^{\i}\theta<\bsigma^{\i+\j}\theta$, since $\sigma$ is order preserving. Now the inequality
$$v(G_{(\i)}(a)) + \bsigma^{\i}\gamma < v(G_{(\i+\j)}(a)) + \bsigma^{\i+\j}\gamma$$
together with $\gamma_1>\gamma$ leads to 
$$v(G_{(\i)}(a)) + \bsigma^{\i}\gamma_1 < v(G_{(\i+\j)}(a)) + \bsigma^{\i+\j}\gamma_1. \footnote{With $\i=(0,1,0 \dots,0)$, the multi-index $(0,0,1,0,\dots,0)$ can not be written as $\i+\j$ and hence we do not claim $\sigma(\gamma_1)<\sigma^2(\gamma_1)$, which is not true in general. }$$ 
Hence $(G,b)$ is in $\sigma$-hensel configuration with $\gamma_1=\gamma(G,b)$.
\end{proof}

Using the above lemma it is straightforward to obtain the following, see ~\cite{ad} for a proof.

\begin{lemma}\label{hensel.imm} Suppose $\ca{K}$
satisfies Axiom $2$ and $G(x)$ is $\sigma$-henselian at $a$.
Suppose also that there is no $b\in K$ with $G(b)=0$ and $v(a-b) = v(G(a))$.
Then there is a pc-sequence $\{a_\rho\}$ in $K$ 
with the following properties:
\begin{enumerate}
\item $a_0=a$ and $\{a_\rho\}$ has no pseudolimit in $K$;
\item $\{v(G(a_\rho))\}$ is strictly increasing, and thus 
 $G(a_\rho) \leadsto 0$;
\item $v(a_{\rho'} -a_\rho)=v\big(G(a_\rho)\big)$ whenever $\rho<\rho'$;
\item for any extension $\ca{K}'=(K',\dots)$ of $\ca{K}$ and $b,c\in K'$ such that $a_\rho \leadsto b$, $G(c)=0$ and $v(b-c)\ge v(G(b))$, we have $a_\rho \leadsto c$.
\end{enumerate}
\end{lemma}

\begin{definition}\label{shensel2} A valued difference field $\ca{K}$ is {\em $\sigma$-henselian\/} if 
for all $(G,a)$ in $\sigma$-hensel configuration there is 
$b \in K$ such that $v(b-a)=\gamma(G,a)$ and $G(b)=0$.
\end{definition}

\begin{lemma}\label{hensax1} If $\ca{K}$ is 
$\sigma$-henselian, then $\ca{K}$ satisfies Axiom $2$.
\end{lemma}

\begin{proof} Assume that $\ca{K}$ is $\sigma$-henselian and 
let $\alpha_0, \dots, \alpha_n \in \bk$, not all zero. Let 
$$G(x)=1+ a_0x + \cdots + a_n\sigma^n(x)  \qquad (\text{all }a_i\in K),$$
where $a_i=0$ if $\alpha_i=0$, and $v(a_i)=0$ with $\bar{a}_i=\alpha_i$ if
$\alpha_i \ne 0$, for $i=0, \dots, n$. It is 
easy to see that $(G,0)$ is in $\sigma$-hensel configuration with $\gamma(G,a)=0$. 
This gives $a \in K$ such that $v(a)=0$ and $G(a)=0$. Then $\bar{a}$ 
is a solution of 
$$1+\alpha_0x + \cdots + \alpha_n \bar{\sigma}^n(x)=0.$$
\end{proof}

\begin{definition}
We say $\{a_\rho\}$ is of
{\em $\sigma$-algebraic type over $K$\/} if $G(b_\rho) \leadsto 0$
for some $\sigma$-polynomial $G(x)$ over $K$ and an equivalent 
pc-sequence $\{b_\rho\}$ in $K$.

If $\{a_\rho\}$ is of $\sigma$-algebraic type over $K$, then a 
{\em minimal $\sigma$-polynomial of  $\{a_\rho\}$ over $K$\/} is a 
$\sigma$-polynomial $G(x)$ over $K$ with the following properties: 
\begin{enumerate}
\item[(i)] $ G(b_\rho) \leadsto 0$ for some pc-sequence 
$\{b_\rho\}$ in $K$, 
equivalent to $\{a_\rho\}$;
\item[(ii)] $ H(b_\rho) \not\leadsto 0$ whenever $H(x)$ 
is a $\sigma$-polynomial 
over $K$ of lower complexity than $G$ and 
$\{b_\rho\}$ is a pc-sequence in $K$ equivalent to $\{a_\rho\}$.
\end{enumerate}
\end{definition}

\noindent
If  $\{a_\rho\}$ is of $\sigma$-algebraic type over $K$, then 
$\{a_\rho\}$ clearly has a minimal $\sigma$-polynomial over $K$.

\begin{lemma}\label{henselconf}
Suppose $\ca{K}$ satisfies Axiom 2. 
Let $\{a_\rho\}$ from $K$ be a pc-sequence of
$\sigma$-algebraic type over $K$ with minimal $\sigma$-polynomial  
$G(x)$ over $K$, and with pseudolimit $a$ in some extension.
Let $\Sigma$ be a finite set of $\sigma$-polynomials $H(x)$ 
over $K$. Then there is a pc-sequence
$\{b_\rho\}$ in $K$, equivalent to $\{a_\rho\}$, such that, 
with $\gamma_\rho:=v(a-a_\rho)$:
\begin{enumerate}
\item[$(1)$] $G(b_\rho) \leadsto 0$ and eventually $v(a-b_\rho)=\gamma_\rho$;
\item[$(2)$] if $H\in \Sigma$ and $H \notin K$, then  
$H(b_\rho) \leadsto H(a)$;
\item[$(3)$] Eventually $(G, b_\rho)$ is in $\sigma$-hensel configuration with $\gamma(G,b_\rho)=\gamma_\rho$;
\item[$(4)$] If $G(a)\neq 0$, then $(G,a)$ is in  $\sigma$-hensel configuration, and 
$\gamma(G,a) > \gamma_\rho$, eventually.
\end{enumerate} 
\end{lemma}
\begin{proof}
Let $G$ have order $n$. We can assume that $\Sigma$ includes
all $G_{(\i)}$.
In the rest of the proof $\i,\j,\l$ range over $\mathbb{N}^{n+1}$.
Since Axiom 2 implies Axiom 1, Theorem~\ref{crucial.result.nonwitt} and its proof
yield an equivalent pc-sequence 
$\{b_\rho\}$ in $K$ such that (1) and (2) hold. The proof of
Theorem~\ref{adjustment1} 
shows we can arrange in addition that
there is $\i \in \mathbb{N}^{n+1}$ such that, eventually,
$$v\big(G(b_\rho)-G(a)\big)= v\big(G_{(\i)}(a)\big)
+\sigma^{\i}\gamma_\rho \leq v\big(G_{(\j)}(a)\big)+\sigma^{\j}\gamma_\rho,$$
for each $\j \neq \i$. 
Now $\left\{v\big(G(b_\rho)\big)\right\}$ is strictly increasing, 
eventually, so $v\big(G(a)\big) > v\big(G(b_\rho)\big)$ eventually, 
and so for $\j \neq \i$: 
$$v\big(G(b_\rho)\big)= v\big(G_{(\i)}(a)\big)
+\sigma^{\i}\gamma_\rho \leq v\big(G_{(\j)}(a)\big)+\sigma^{\j}\gamma_\rho,
\quad \text{eventually}.$$  
We claim that $|\i|=1$. Let $|\j|=1$ with $G_{(\i)}\ne 0$, and let $\k>\j$;
our claim will then follow by deriving 
$$v\big(G_{(\j)}(a)\big) + \sigma^{\j}\gamma_\rho < v\big(G_{(\k)}(a)\big) + 
\sigma^{\k}\gamma_\rho, \quad \text{eventually}.$$ The proof of
Theorem~\ref{adjustment1} with $G_{(\j)}$ in the role of $G$ shows that 
we can arrange that our sequence $\{b_\rho\}$ also satisfies
$$v\big(G_{(\j)}(b_\rho)-G_{(\j)}(a)\big) \le  v\big(G_{(\j)(\l)}(a)\big)
+\sigma^{\l}\gamma_\rho, \quad\text{eventually}$$
for all $\l$ with $|\l|\ge 1$. Since 
$v\big(G_{(\j)}(b_\rho)\big)=v\big(G_{(\j)}(a)\big)$ eventually, this yields
$$v\big(G_{(\j)}(a)\big)\leq v\big(G_{(\j)(\l)}(a)\big)
+\sigma^{\l}\gamma_\rho, \quad\text{eventually}$$
for all $\l$ with $|\l|\ge 1$, hence for all such $\l$,
$$v\big(G_{(\j)}(a)\big)\leq
 v{{\j+ \l}\choose \i}
+v\big(G_{(\j+ \l)}(a)\big)+\sigma^{\l}\gamma_\rho, \quad\text{eventually} $$
For $\l$ with $\j+\l=\k$, this yields
$$v\big(G_{(\j)}(a))\leq v{{\k\choose \j}}
+v\big(G_{(\k)}(a)\big)+\sigma^{\k-\j}\gamma_\rho, \quad\text{eventually}.$$
Since $\cha(\bk)=0$,
$$v\big(G_{(\j)}(a))\leq v\big(G_{(\k)}(a)\big)+\sigma^{\k-\j}\gamma_\rho,  \quad\text{eventually, hence}  $$
$$v\big(G_{(\j)}(a)\big)+\sigma^{\j}\gamma_\rho<v(G_{(\k)}(a))+
\sigma^{\k}\gamma_\rho, \quad\text{eventually}. $$
Thus $|\i|=1$ as claimed. Then we obtain (3) and (4) by the above inequalities together with the fact that $G$ is a minimal
$\sigma$-polynomial of $b_\rho$. 
\end{proof}

\section{Immediate Extensions}\label{max.imm.ext}

\noindent
Throughout this section
$\ca{K}=(K,\Gamma,\bk; v,\pi)$ is a valued difference field satisfying Axiom 2. 
Note that immediate 
extensions of $\ca{K}$ also satisfy Axiom 2. We let $K$ stand for $\ca{K}$ 
when the meaning is clear from the context. The results of this section contain precisely the same conclusions as in the results of the corresponding sections of ~\cite{BMS}, ~\cite{ad} and ~\cite{vdca}. Having proved Lemmas~\ref{newton2} and ~\ref{henselconf} for the general context of this paper, proofs in ~\cite{BMS}, ~\cite{ad} and ~\cite{vdca} remain intact. Therefore we present the results without proof.

\begin{definition}
A pc-sequence $\{a_\rho\}$ from $K$ is said to be of 
{\em $\sigma$-transcendental type over $K$\/}
if it is not of $\sigma$-algebraic type over $K$, that is, 
$G(b_\rho) \not\leadsto 0$ for each $\sigma$-polynomial $G(x)$
over $K$ and each equivalent pc-sequence $\{b_\rho\}$ from $K$. 
\end{definition}

\noindent
In particular, such a pc-sequence cannot have a pseudolimit in $K$.
For the proofs of next two lemmas see \cite{BMS} or \cite{vdca}. 

\begin{lemma}\label{ext.with.pseudolimit} 
Let $\{a_\rho\}$ from $K$ be a pc-sequence of $\sigma$-transcendental type
over $K$.  Then  $\ca{K}$ has an
immediate extension $(K\langle a \rangle, \Gamma,\bk; v_a, \pi_a)$
such that: \begin{enumerate} 
\item[$(1)$] $a$ is $\sigma$-transcendental over $K$ and 
$a_\rho \leadsto a$;
\item[$(2)$] for any extension $(K_1,\Gamma_1,\bk_1;v_1,\pi_1)$ of 
$\ca{K}$ and any $b\in K_1$ with
$a_\rho \leadsto b$ there is a unique
embedding 
$$(K\langle a \rangle, \Gamma,\bk; v_a, \pi_a)\ \longrightarrow\ 
(K_1,\Gamma_1,\bk_1;v_1,\pi_1)$$ 
over  $\ca{K}$ that sends $a$ to $b$. 
\end{enumerate} 
\end{lemma}

\begin{lemma}\label{imm.alg.ext} 
Let $\{a_\rho\}$ from $K$ be a pc-sequence of $\sigma$-algebraic type over
$K$, with no pseudolimit in $K$. Let $G(x)$ be a minimal 
$\sigma$-polynomial of $\{a_\rho\}$ over $K$. Then $\ca{K}$ 
has an immediate extension 
$(K\langle a \rangle, \Gamma,\bk; v_a, \pi_a)$ such that
\begin{enumerate}
\item[$(1)$] $G(a)=0$ and $a_\rho \leadsto a$;
\item[$(2)$] for any extension $(K_1, \Gamma_1,\bk_1; v_1, \pi_1)$ of 
$\ca{K}$ and any $b\in K_1$ with $G(b)=0$ and
$a_\rho \leadsto b$ there is a unique
embedding 
$$(K\langle a \rangle, \Gamma,\bk; v_a, \pi_a) \longrightarrow\ 
(K_1, \Gamma_1,\bk_1; v_1, \pi_1)$$ 
over  $\ca{K}$ that sends $a$ to $b$. 
\end{enumerate} 
\end{lemma}

\noindent
We note the following consequences of 
Lemmas~\ref{ext.with.pseudolimit} and ~\ref{imm.alg.ext}:

\begin{corollary}\label{ctra} Let $a$ from some extension of
$\ca{K}$ be $\sigma$-algebraic over $K$ and let 
$\{a_\rho\}$ be a  pc-sequence in $K$ such that $a_\rho \leadsto a$.
Then $\{a_\rho\}$ is of $\sigma$-algebraic type over $K$.
\end{corollary}

\begin{corollary} \label{immediate.sigma-immediate}
Suppose $\ca{K}$ is finitely ramified. Then $\ca{K}$ as a valued field has 
a proper immediate extension if and only if
$\ca{K}$ as a valued difference field has a proper immediate extension.
\end{corollary}

\noindent
We say that $\ca{K}$ is 
{\em $\sigma$-algebraically maximal\/} if it
has no proper immediate $\sigma$-algebraic extension, and we say it is
{\em maximal\/} if it has no proper immediate extension. 
Corollary~\ref{ctra} and Lemmas~\ref{imm.alg.ext} and ~\ref{hensel.imm} yield:

\begin{corollary}\label{crucialhens}
\begin{enumerate}
\item[$(1)$] $\ca{K}$ is $\sigma$-algebraically maximal 
if and only if each pc-sequence in $K$ of $\sigma$-algebraic type over $K$ 
has a pseudolimit in $K$;
\item[$(2)$] if $\ca{K}$ satisfies Axiom $2$ and is
$\sigma$-algebraically maximal, then $\ca{K}$ is 
$\sigma$-henselian.
\end{enumerate}
\end{corollary}

\noindent
It is clear that $\ca{K}$ has
$\sigma$-algebraically maximal immediate $\sigma$-algebraic extensions, and 
also maximal immediate extensions. Using the next lemma, we will show that if $\ca{K}$ satisfies Axiom 2 both kinds of extensions are unique up to isomorphism.

\begin{lemma} Suppose $\ca{K}$ satisfies Axiom 1 and
$\ca{K}'$ satisfies Axiom 2.
Let $\ca{K}'$ be $\sigma$-algebraically maximal
extension of $\ca{K}$ and
$\{a_\rho\}$ from $K$ be a pc-sequence of $\sigma$-algebraic type over
$K$, with no pseudolimit in $K$, and with minimal 
$\sigma$-polynomial $G(x)$ over $K$. Then there exists $b\in K'$
such that $a_\rho \leadsto b$ and $G(b)=0$.
\end{lemma}
\begin{proof} Lemma~\ref{imm.alg.ext} provides a pseudolimit $a\in K'$ of
$\{a_\rho\}$. Take a pc-sequence $\{b_\rho\}$ in $K$ 
equivalent to $\{a_\rho\}$ with the properties listed in 
Lemma~\ref{henselconf}. 
Since $\ca{K}'$ satisfies Axiom 2, it is $\sigma$-henselian 
and hence there is $b\in K'$ such that
$$ v(a-b)=\gamma(G,a)\ \text{  and }\ G(b)=0.$$ 
Note that $a_\rho \leadsto b$ since
$\gamma(G,a) > v(a-a_\rho)=\gamma_\rho$ eventually.
\end{proof} 

\noindent
Together with Lemmas~\ref{ext.with.pseudolimit} and ~\ref{imm.alg.ext}
this yields:

\begin{theorem}\label{unique.max.imm.ext} Suppose $\ca{K}$ 
satisfies Axiom $2$. Then all its maximal immediate extensions
are isomorphic over
$\ca{K}$, and all its $\sigma$-algebraically maximal immediate 
$\sigma$-algebraic extensions are isomorphic over
$\ca{K}$.
\end{theorem}

It is worth mentioning that the above result fails if Axiom 2 is not assumed, see ~\cite{vdca}. A minor variant of these results will be needed in proving our model theoretic conclusions, its proof is straightforward. Let $|X|$ denote the cardinality of a set $X$, and let $\kappa$ be a 
cardinal.

\begin{lemma} Suppose $\ca{E}=(E, \Gamma_E, \dots)\le \ca{K}$ satisfies Axiom 1 and $\ca{K}$ is $\sigma$-henselian, and $\kappa$-saturated with
$\kappa>|\Gamma_E|$. 
Let $\{a_\rho\}$ from $E$ be a pc-sequence of $\sigma$-algebraic type over
$E$, with no pseudolimit in $E$, and with minimal
$\sigma$-polynomial $G(x)$ over $E$. Then there exists $b\in K$
such that $a_\rho \leadsto b$ and $G(b)=0$.
\end{lemma}

\noindent
In combination with Lemmas~\ref{ext.with.pseudolimit} and 
~\ref{imm.alg.ext} this yields:

\begin{corollary}\label{immsat} If $\ca{E}=(E, \Gamma_E, \dots)\le \ca{K}$
satisfies Axiom $2$, and $\ca{K}$ is 
$\sigma$-henselian, and  $\kappa$-saturated with
$\kappa>|\Gamma_E|$, then any maximal immediate 
extension of $\ca{E}$ can be embedded in $\ca{K}$ over $\ca{E}$.
\end{corollary}

\begin{section}{The Embedding Theorem}\label{et}

Theorem~\ref{embed11}, the main result of the paper,
will give us a criterion for elementary equivalence between $\sigma$-henselian valued difference fields
of residue characteristic zero and also a relative quantifier elimination result. We now present the notion of {\em rv-structure} for a valued field which will be needed for relative quantifier elimination. This notion was introduced in~\cite{basarabkuhl} building up on the notion of {\em additive multiplicative congruences}.

\medskip\noindent
{\bf RV-Structure.}
Let $\ca{K}$ be a valued field. The rv-sort 
for $K$ is the imaginary sort $RV=K^{\times}/1 + \mathfrak{m}$, which is a multiplicative group. The associated canonical surjection is denoted by $rv$ and we extend it to $K$ by setting $rv(0) \colon=\infty$. The  subgroup $\mathcal{O}^{\times}/1 + \mathfrak{m}$ of $RV$ is exactly $\bk^{\times}$. Two elements $a, b \in K$ are equivalent modulo $1 + \mathfrak{m}$ if and only if $v(a-b)>v(a)=v(b)$ and so the map
\begin{align*}
v_{rv} \colon& RV \to \Gamma \\
&a(1+\mathfrak{m}) \mapsto v(a)
\end{align*}
is well-defined. 
There is also a partial addition on the rv-sort. For $a,b \in K$ with $v(a+b)=\min\{v(a),v(b)\}$, define
$$rv(a)+rv(b)\colon=rv(a+b).$$
It is clear that this partial addition is well-defined and we can extend this addition to all of $RV$ by setting
$rv(a)+rv(b):=\infty$ whenever $v(a+b) \neq\min\{v(a),v(b)\}$. We denote this map by $\bigoplus$.

\medskip\noindent
If $\ca{K}$ is a valued difference field, its distinguished automorphism $\sigma$ fixes $\mathfrak{m}$ set-wise, and so the map
$\sigma_{rv} : x(1 + \mathfrak{m}) \to    rv(\sigma(x))$, is also well-defined. Moreover the induced maps  
$\sigma_{rv}:\Gamma \to \Gamma$ and $\sigma:\Gamma \to \Gamma$ are the same. We consider the rv-sort of $\ca{K}$ as first-order structure in the language $\mathcal{L}_{rv}:=\{.,^{-1},\oplus, 1, v_{rv}, \sigma_{rv}\}$ and refer to it as the {\em difference rv structure} of $\mathcal{K}$. Replacing the addition, multiplication and the difference operator on $K$ by the corresponding operations on the rv-sort, we may consider a $\sigma$-polynomial $F(x)$ over $K$ as a function on the rv-sort, also denoted $F(x)$. It is worth noting that the value group and residue field is interpretable in the rv-structure both for valued fields and valued difference fields, see Proposition 9.3 of~\cite{Pal}.

\medskip\noindent
{\bf The Main Result.}
In this subsection we consider $4$-sorted structures
$$ \ca{K}=\big(K, \Gamma, \bk, RV; v, rv, v_{rv}, \pi)$$
where $\big(K, \Gamma, \bk; v, \pi \big)$ is a valued difference
field and $RV$ is its RV sort.
Such a structure will be called an
{\em rv-valued difference field\/}. Any subfield $E$ of $K$ is
viewed as a valued subfield of $\ca{K}$ with valuation ring
$\ca{O}_E:=\ca{O}\cap E$.


\medskip\noindent
A {\em good substructure\/} of 
$\ca{K}=(K, \Gamma, \bk, RV; v, rv, v_{rv}, \pi)$ is a triple
$\ca{E}=(E, \Gamma_{\ca{E}}, \bk_{\ca{E}}, RV_{\ca{E}})$ such that \begin{enumerate}
\item $E$ is a difference subfield of $K$, 
\item $\Gamma_{\ca{E}}$ is an ordered abelian subgroup of $\Gamma$ 
with $v(E^\times)\subseteq \Gamma_{\ca{E}}$,
\item $\bk_{\ca{E}}$ is a difference subfield of $\bk$ with 
$\pi(\ca{O}_E)\subseteq \bk_{\ca{E}}$,
\item $RV_{\ca{E}}$ is a subgroup of $RV$ with $rv(E) \subseteq RV_{\ca{E}}$.
\end{enumerate}

 For good substructures $\ca{E}_1=(E_1, \Gamma_1, \bk_1, RV_1)$ and
$\ca{E}_2=(E_2, \Gamma_2, \bk_2, RV_2)$ of $\ca{K}$, we define 
$\ca{E}_1\subseteq \ca{E}_2$ to mean that 
$E_1 \subseteq E_2,\ \Gamma_1 \subseteq \Gamma_2,\ \bk_1 \subseteq \bk_2$ and $RV_1 \subseteq RV_2$.
If $E$ is a difference subfield of $K$ with 
$\ac(E)=\pi(\ca{O}_E)$, then 
$\big(E, v(E^\times), \pi(\ca{O}_E)\big)$ is a good substructure of $\ca{K}$, 
and if in addition $F\supseteq E$ is a difference subfield
of $K$ such that $v(F^\times)=v(E^\times)$, then $\ac(F)=\pi(\ca{O}_F)$.
Throughout this subsection 
$$\ca{K}=(K, \Gamma, \bk, RV; v, rv, v_{rv}, \pi), \qquad \ca{K}'=(K', \Gamma', \bk', RV'; v', rv', v'_{rv}, \pi')$$ 
are rv-valued difference fields, with 
valuation rings $\ca{O}$ and $\ca{O}'$, and
$$\ca{E}=(E,\Gamma_{\ca{E}}, \bk_{\ca{E}}, RV_{\ca{E}}), \qquad \ca{E'}=(E',\Gamma_{\ca{E}'},
\bk_{\ca{E}'}, RV_{\ca{E}'})$$ are good substructures of
$\ca{K}$, $\ca{K'}$ respectively. To avoid too complicated notation we let $\sigma$ 
denote the difference operator of each of $K, K', E, E'$, and put 
$\ca{O}_{E'}:= \ca{O}'\cap E'$.

 A {\em good map\/} $\mathbf{f}: \ca{E} \to \ca{E'}$
is a quadruple $\mathbf{f}=(f, f_{\val}, f_{\res}, f_{rv})$ consisting
of a difference field isomorphism
$f:E \to E'$, an ordered group isomorphism $f_{\val}:\Gamma_{\ca{E}} \to
\Gamma_{\ca{E}'}$, a difference field isomorphism $f_{\res}: \bk_{\ca{E}} \to
\bk_{\ca{E}'}$ and a group isomorphism $f_{rv}: RV_{\ca{E}} \to RV_{\ca{E}'}$ such that
\begin{enumerate}
\item[(i)] $f_{\val}(v(a))=v'(f(a))$ for all $a \in E^\times$, and
$f_{\val}$ is elementary as a partial map between $\Gamma$ and $\Gamma'$; 
\item[(ii)]
$f_{\res}(\pi(a))=\pi'(f(a))$ for all $a \in E$, and
$f_{\res}$ is elementary as a partial map between $\bk$ and $\bk'$;
\item[(iii)] $f_{rv}(rv(a))=rv'(f(a))$ for all $a \in E^{\times}$ and $f_{rv}$ is elementary as a partial map between $RV_{\ca{E}}$ and $RV_{\ca{E}'}$. 
\end{enumerate}
Let $\mathbf{f}: \ca{E} \to \ca{E'}$ be a good map as above. Then the 
field part $f: E \to E'$ of $\af$ is a valued difference field isomorphism. Moreover
$f_{\val}$, $f_{\res}$ and $f_{rv}$ agree on $v(E^\times)$, $\pi(\ca{O}_E)$ and $rv(E)$ with
the corresponding maps induced by $f$. 
We say that a good map $\ag= (g, g_{\val}, g_{\res}, g_{rv}) : \ca{F} \to \ca{F'}$ 
{\em extends\/} $\af$ if $\ca{E}\subseteq \ca{F}$, $\ca{E'}\subseteq \ca{F'}$, 
and $g$, $g_{\val}$, $g_{\res}$, $g_{rv}$ extend $f$, $f_{\val}$, $f_{\res}$, $g_{rv}$ respectively.  
The {\em domain\/} of $\mathbf{f}$ is $\ca{E}$.

\noindent
We say that $\ca{E}$ satisfies Axiom 1 (respectively, Axiom 2) 
if the valued difference subfield $(E, v(E^\times),\pi(\ca{O}_E);\dots)$ of 
$\ca{K}$ does. Likewise, we say that $\ca{E}$ is $\sigma$-henselian if this valued difference subfield of $\ca{K}$ is. 

Regular elements will also play a role in the main theorem, for value group extensions. The same technique has been used in ~\cite{Pal} with a different name\footnote{{\em Generic elements} from ~\cite{Pal} correspond to regular elements.}. The next result is the analogue of Lemma 8.8 from~\cite{Pal}, whose proof can be generalized to our context easily using Lemma~\ref{regforpc}.

\begin{lemma}\label{regforvalue}
Let $\ca{K}$ be a $\kappa^{+}$-saturated valued difference field which satisfies Axiom 1. Suppose that $\ca{E}$ is a good substructure of $\ca{K}$ with $\bk_{\ca{E}} < \kappa$. Let $\gamma \in \Gamma$, $\gamma \notin v(E^{\times})$.
Then
\begin{itemize}
\item[(i)] there is $a \in K$ with $v(a)=\gamma$ which is generic over $E$;
\item[(ii)] The value group of $E\langle a \rangle$ is $v(E^\times)\langle \gamma \rangle$ and its residue field is the same as the residue field of $E$;
\item[(iii)] if $b\in K$ is generic over $E$ with $v(a)=\gamma'$ and the ordered difference groups $v(E^\times)\langle \gamma \rangle$ and $v(E^\times)\langle \gamma' \rangle$ are isomorphic via the map sending $\gamma$ to $\gamma'$ then there is a valued difference field isomorphism between $E\langle a \rangle$ and $E\langle b \rangle$ sending $a$ to $b$.
\end{itemize}
\end{lemma}

\begin{theorem}\label{embed11} Suppose $\cha(\bk)=0$, $\ca{K}$, $\ca{K}'$ 
satisfy Axiom $2$ and are $\sigma$-henselian. Then any good map
$\ca{E} \to \ca{E}'$ is a partial elementary map between $\ca{K}$ and 
$\ca{K}'$.
\end{theorem}    
\begin{proof} The theorem holds trivially for $\Gamma=\{0\}$, so assume 
that $\Gamma\ne \{0\}$. 
Let $\mathbf{f}=(f, f_{\val}, f_{\res}, f_{rv}): \ca{E} \to \ca{E}'$
be a good map. By passing to suitable elementary extensions of 
$\ca{K}$ and $\ca{K}'$ we may assume that
$\ca{K}$ and $\ca{K}'$ are $\kappa$-saturated, where 
$\kappa$ is an uncountable cardinal such that 
$|\bk_{\ca{E}}|,\ |\Gamma_{\ca{E}}| < \kappa$.  
We say thatl a good substructure $\ca{E}_1=(E_1,\bk_1, \Gamma_1)$ of $\ca{K}$ 
{\em small\/} if $|\bk_1|,\ |\Gamma_1|<\kappa$.
We shall prove that the good maps with small domain form a back-and-forth
system between $\ca{K}$ and $\ca{K}'$, which suffices to obtain the theorem. In other words, 
we shall prove that under the 
present assumptions on $\ca{E}$, $\ca{E}'$ and $\mathbf{f}$, there is   
for each $a \in K$ a good map $\ag$ extending $\af$ such that
$\ag$ has small domain $\ca{F}=(F,\dots)$ with $a\in F$. We achieve this by appropriately iterating Corollary~\ref{immsat}, and the extension 
procedures described below which correspond to extending the extending the residue field and value group. We will use results from ~\cite{ad} to extend the residue field have been used in different contexts in ~\cite{vdca} and ~\cite{Pal}. 

\medskip\noindent
(1) {\em Given $\alpha\in \bk$, arranging that $\alpha\in \bk_{\ca{E}}$}.
By saturation and the definition of ``good map'' this can be achieved without 
changing $f$, $f_{\val}$, $f_{rv}$, $E$, $\Gamma_{\ca{E}}$, $RV_{\ca{E}}$
by extending $f_{\res}$ to a partial elementary map between $\bk$ and $\bk'$ 
with $\alpha$ in its domain. In the same way, we obtain the next two results. 
\medskip\noindent
(2) {\em Given $\gamma\in \Gamma$, arranging that $\gamma\in \Gamma_{\ca{E}}$}.

\medskip\noindent
(3) {\em Given $r\in RV$, arranging that $r \in RV_{\ca{E}}$}.

\medskip\noindent
(4) {\em Arranging $\bk_{\ca{E}}=\pi(\ca{O}_E)$}. Suppose 
$\alpha\in \bk_{\ca{E}},\ \alpha\notin \pi(\ca{O}_E)$; set 
$\alpha':= f_{\res}(\alpha)$. 

If $\alpha$ is $\bar{\sigma}$-transcendental over $\pi(\ca{O}_E)$,
we pick $a\in \ca{O}$ and $a'\in \ca{O}'$ such
that $\bar{a}=\alpha$ and $\bar{a'}=\alpha'$, and then
Lemma 2.5 from ~\cite{ad} yields a 
good map $\ag=(g, f_{\val}, f_{\res})$ with small domain
$(E \langle a \rangle, \Gamma_{\ca{E}}, \bk_{\ca{E}})$ 
such that $\ag$ extends $\af$ and $g(a)=a'$.

Next, assume that $\alpha$ is $\bar{\sigma}$-algebraic
over $\pi(\ca{O}_E)$. Let $G(x)$ be a $\sigma$-polynomial over $\ca{O}_E$
such that $\bar{G}(x)$ is a minimal
$\bar{\sigma}$-polynomial of $\alpha$ over $\pi(\ca{O}_E)$ and 
has the same complexity
as $G(x)$. Pick $a\in \ca{O}$ such
that $\bar{a}=\alpha$. Then $G$ is $\sigma$-henselian at $a$. So we have $b
\in \ca{O}$ such that $G(b)=0$ and $\bar{b}=\bar{a}=\alpha$. 
Likewise, we obtain $b'\in \ca{O}'$ such that $f(G)(b')=0$ and 
$\bar{b'}=\alpha'$, where $f(G)$ is
the difference polynomial over $E'$ that corresponds to $G$ under $f$.
By Lemma 2.6 of ~\cite{ad}, we obtain a good map extending $\af$ 
with small domain $(E \langle b \rangle,  \Gamma_{\ca{E}}, \bk_{\ca{E}})$ and 
sending $b$ to $b'$. 

\bigskip\noindent
Iterating (4) we may assume that
$\bk_{\ca{E}}=\pi(\ca{O}_E)$. We refer from now on to
$\bk_{\ca{E}}$ as the {\em residue difference field\/} of $E$. 
Since $\ca{K}$ satisfies Axiom 2, appropriately iterating (1) and (4)  we may assume in addition that $\ca{E}$ satisfies Axiom 2. 

\bigskip\noindent
(5) {\em Arranging $\Gamma_{\ca{E}}=v(E^\times)$}. 

Suppose $\gamma\in \Gamma_{\ca{E}}$, $\gamma \notin v(E^\times)$. By Lemma~\ref{regforvalue}, we can take $b \in K$
such that $v(b)=\gamma$ and $b$ is regular over $K$. By (3), we may assume $r=rv(b) \in RV_{\ca{E}}$. Let $r'=f_{rv}(r) \in \ca{K}'$ and pick $b' \in \ca{K}'$ with $rv(b')=r'$. Then $b'$ is generic over $E'$ and we can extend $\af$ by mapping $b$ to $b'$, again by Lemma~\ref{regforvalue}.

\bigskip\noindent
Iterating (5) we assume in the rest of the proof that  
$\Gamma_{\ca{E}}=v(E^\times)$. This condition is actually
preserved in the earlier extension procedures (4). We refer from now on to
$\Gamma_{\ca{E}}$ as the {\em value group\/} of $E$. 
Note also that in the extension procedure (5) the
residue difference field does not change.

\bigskip\noindent
Now let $a \in K$ be given. We want to extend $\af$ to a good map
whose domain is small and contains $a$. By our previous remarks  
$\bk_{\ca{E}}=\pi(\ca{O}_E)$, $\Gamma_{\ca{E}}=v(E^\times)$, and 
$\ca{E}$ satisfies Axiom 2.
Appropriately iterating and alternating the above 
extension procedures
we can arrange in addition that 
$E\langle a \rangle$ is an immediate
extension of $E$. Let $\ca{E} \langle a \rangle$ be the valued
difference subfield of $\ca{K}$ that has $E\langle a \rangle$ as underlying
difference field. By Corollary~\ref{immsat}, $\ca{E} \langle a \rangle$ has a maximal immediate 
valued difference field extension $\ca{E}_1\le \ca{K}$. Then
$\ca{E}_1$ is a maximal immediate extension of $\ca{E}$ as well. 
Applying Corollary~\ref{immsat} 
to $\ca{E}'$ and using 
Theorem~\ref{unique.max.imm.ext}, we can extend
$\af$ to a good map with domain $\ca{E}_1$, construed here as a good
substructure of $\ca{K}$ in the obvious way. Of course, 
$a$ is in the underlying difference field of $\ca{E}_1$.
\end{proof}

\end{section}

\begin{section}{Equivalence and Relative Quantifier Elimination}\label{cqe1}

\noindent
Here we state the model theoretic consequences of the Theorem~\ref{embed11}. All these consequences are obtained by standard model theoretic considerations which has been done in detail repeatedly in the context of valued difference fields, see~\cite{BMS}, ~\cite{ad}, ~\cite{vdca}, ~\cite{Pal}. We use the symbols $\equiv$ and $\preceq$ for the 
relations of elementary equivalence and being an elementary submodel,
in the setting of many-sorted structures.
In this section
$$\ca{K}=(K, \Gamma, \bk, RV; \dots), \qquad \ca{K}'=(K', \Gamma', \bk', RV'; \dots)$$
are rv-valued difference fields of 
residue characteristic $0$ that satisfy Axiom $2$ and are $\sigma$-henselian. We consider them as $\ca{L}$-structures where $\ca{L}$ is the 4-sorted language described in the previous section. Note that $\Gamma$ and $\bk$ are interpretable in $RV$, hence:

\begin{theorem}\label{comp00}  
$\ca{K} \equiv \ca{K}'$ if and only if $RV \equiv RV'$.
\end{theorem}

\noindent
The Hahn difference field
$\bk((t^\Gamma))$ can be expanded to an rv-valued difference field in the natural way and, by the above result, it is elementarily equivalent to $\ca{K}$. 

\begin{theorem}\label{comp01} Let 
$\ca{E}=(E, \Gamma_E, \bk_E, RV_E;\dots)$ be a $\sigma$-henselian rv-valued difference subfield of $\ca{K}$  such that $RV_E\preceq RV$. Then $\ca{E} \preceq \ca{K}$.
\end{theorem}

\begin{theorem}\label{qe} Let $\ca{T}$ be the $\ca{L}$-theory of $\sigma$-henselian valued difference fields with residue charateristic 0. Then every $\ca{L}$-formula $\phi$ is equivalent (modulo $\ca{T}$) to an $\ca{L}$-formula $\psi$ in which all occurences of field variables are free.
\end{theorem}

\end{section}

\medskip\noindent
{\bf Cross-section and Angular Component Maps}: Let $\ca{K}=(K, \Gamma, \bk;v,\pi)$ be a valued field. A cross-section map  is a group homomorphism $c:\Gamma \to K^{\times}$ such that $v(c(\gamma))=\gamma$, for all $\gamma \in \Gamma$. An {\em angular component} map is a multiplicative group homomorphism $ac: K^{\times} \to \bk^{\times}$ which agrees with the residue class map $\pi$ on $\ca{O}\setminus \fr{m}$. We can extend an angular component map to $K$ by setting $ac(0)=0$. In the presence of a {\em cross-section} the rv-sort is interpretable in $\ca{K}=(K, \Gamma, \bk; v,\pi, c)$ by taking 
$rv:K^{\times} \to k^{\times} \times \Gamma$ with $rv(a)=(\pi(a/c(v(a))), v(a))$. Similarly, when $K$ is equipped with an angular component map the rv-sort is interpretable in $\ca{K}=(K,\Gamma,\bk;v,\pi,ac)$ via $rv(a)=(ac(a),v(a))$. Evert valued field has elementary extension which admits cross-section and an angular component map.

If $\ca{K}$ is a valued difference field we ask that cross-section and angular component maps are compatible with the distinguished automorphism $\sigma$. Therefore it seems quite unlikely that an arbitrary valued difference field admits a cross-section or an angular component map in an elementary extensions\footnote{The first author would like to thank Koushik Pal for bringing up this point, which has been neglected in ~\cite{vdca}.}. However, many natural examples of valued difference fields can be equipped with cross-section and angular component maps and so it is worthwhile stating the consequences of the above results in the presence of a crosss-section or an angular component map\footnote{It is also possible to obtain these maps with additional assumptions. For example if the value difference group is flat as a $\mathbb Z [\sigma]$-module then the field admits a cross-section in an elementary extension.}.

\begin{theorem}\label{comp01} Let $\ca{K}, \ca{K'}$ be $\sigma$-henselian valued difference fields with residue characteristic zero which are either both equipped with a cross-section or both equipped with an angular component map. Then $\ca{K}\equiv \ca{K}'$ if and only if $\bk\equiv\bk'$ and $\Gamma\equiv\Gamma'$.
\end{theorem}

\begin{theorem}\label{qe2} Let $\ca{L}_c$ and $\ca{L}_{ac}$ be the expansions of the 3-sorted language of valued difference fields with a cross-section and with an angular component map respectively.  Let $\ca{T}_c$ and $\ca{T}_{ac}$ be the theories of $\sigma$-henselian valued difference fields with residue charateristic 0 in the language $\ca{L}_c$ and $\ca{L}_{ac}$ respectively.  Then every $\ca{L}_c$-formula $\phi$ is equivalent (modulo $\ca{T}_c$) to an $\ca{L}_c$-formula $\psi$ in which all occurences of field variables are free and every $\ca{L}_{ac}$-formula $\phi$ is equivalent (modulo $\ca{T}_{ac}$) to an $\ca{L}_{ac}$-formula $\psi$ in which all occurences of field variables are free.
\end{theorem}

\section{Applications to Transseries}\label{trans}

We refer the reader to ~\cite{vdHLN} for all definitions, conventions and basic facts regarding transseries. Let $\mathbb{T}$ be a field of transseries (grid based or well based) in $x$. Monomials of $\mathbb{T}$ (with coefficient 1) form an ordered multiplicative group. Let $\Gamma$ be this group monomials, seen as an additive group and equipped with the reverse of the ordering on monomials. Then we define the valuation of an element $f(x) \in \mathbb{T}$ as the minimum (in the sense of the reversed ordering) of the support of $f(x)$. So we have 
$$v(e^x)<<v(x)<<v(log(x))<<v(1)=0<v(x^{-1})<<v(e^{-x})$$
where $\gamma<<\gamma'$ is a shorthand for $n\gamma<\gamma'$ for all $n>0$. Note that the residue field of $\mathbb T$ is isomorphic to the field of constants of $\mathbb T$ which is a real closed. Clearly $\mathbb T$ admits a cross-section map; simply send $\gamma \in \Gamma$ to the monomial with coefficient $1$ and valuation $\gamma$. 

\medskip\noindent Let $g(x)$ be an infinite positive element of $\mathbb T$. Then we obtain an automorphism of $\mathbb T$ via sending $f(x)$ to $f(g(x))$ which is obtained from $f(x)$ by uniformly substituting $x$ by $g(x)$. By proposition 5.10 of~\cite{vdHLN} such automorphisms are {\em asymptotic} which, in terms of the valuation, means that they fix the valuation ring of $\mathbb T$ setwise.  Now let $\sigma$ be the automorphism of $\mathbb T$ given by right composition with $g(x)=x+1$, and consider $\mathbb T$ as a valued difference field equipped with a cross-section $(\mathbb T, \Gamma, \bk; v, \pi, c)$ with distinguished automorphism $\sigma$. Note that $\bar{\sigma}$ is the identity, and hence the equation $$\bar{\sigma}(x)-x+1=0$$ has no solution in $\bk$. On the contrary, the action of $\sigma$ on the value group is rather complex;
$$v(\sigma(e^x))=v(e^x), \quad v(\sigma(e^{xlogx}))=v(e^{xlogx})+v(x), \quad v(\sigma(e^{x^2})=v(e^{x^2})+v(e^x)$$
and as such does not fit in any of contexts studied in ~\cite{BMS}, ~\cite{vdca} and ~\cite{Pal}.

We now introduce a coarsening of $v$, whose residue difference field will be linear difference closed. Let $$\Delta:=\{\gamma \in \Gamma: v(e^x)<<\gamma<<v(e^{-x})\}.$$
Then $\Delta$ is a convex subgroup of $\Gamma$ and moreover if $v(f(x)) \in \Delta$ then so is $v(\sigma(f(x))$. Therefore we obtain a valued difference field equipped with a cross-section $$\mathbb{T}_\Delta:=(\mathbb T, \Gamma_w, \bk_w;w,\pi_w, c_w)$$
where $w: \mathbb{T}^{\times} \to \Gamma_w=\Gamma/\Delta$ is the coarsening of $v$ by $\Delta$. Let $K_w$ be the difference subfield of $\mathbb T$ which consists of elements $f(x)$ whose support is contained in $\Delta$. Then $K_w$ is isomorphic to $\bk_w$ via the restriction of the residue class map $\pi_w$. In order to show that $K_w$ is linear difference closed we will use the differential operator $\partial$ on $\mathbb T$, as introduced in Chapter 5 of~\cite{vdHLN}, which has a functional inverse $\int$ and is compatible with composition and exponentiation. Next we list a few conclusions from Proposition 5.11 of~\cite{vdHLN} and its proof.
\begin{lemma}\label{translem}
For all $f(x) \in K_w$ we have:
\begin{itemize}
 \item[(i)] $\partial f(x) \in K_w$; 
 \item[(ii)] $f(x+1)=f(x)+\partial f(x)+ \partial^2f(x)/2!+\partial^3f(x)/3!+\cdots$;
 \item[(iii)] $v(\partial logf(x))>0$.
\end{itemize}
\end{lemma} 

\medskip\noindent
Hence $\sigma(f(x))=e^{\partial}f(x)$ where $e^\partial$ is seen as an operator in $K_w[[\partial]]$. Now let $h_0,\dots, h_n \in K_w$, with $h_i \neq 0$ for some $i$, and consider the linear difference equation 
$$h_0+h_1\sigma(f(x))+\cdots+h_n\sigma^n(f(x))=1$$
We represent this equation as an (infinite order) linear differential equation:
$$Lf(x)=1$$
where $L=h_0+h_1e^\partial+\cdots+h_ne^{n\partial} \in K_w[[\partial]]$ is nonzero. A formal inverse $L^{-1}$, of $L$ can be found in the ring $K_w[[\partial]][\int]$. Note that it is possible to have the constant term of $L$ equal to zero and then one would need the inverse of $\partial$ (which is $\int$) to find $L^{-1}$. See {\em flat discrete summation} in ~\cite{vdHmeta} for a specific example worked out in detail. By part (iii) of Lemma~\ref{translem}, the formal operator $L^{-1}$ indeed acts on $K_w$ and so we can find a solution to the equation $Lf(x)=1$ by choosing $f(x)=L^{-1}(1) \in K_w$. Therefore $K_w$ is linear difference closed. 

\begin{remark}\label{t1} The fact that $K_w$ is linear difference closed is actually implicit in Chapter 7 of~\cite{vdHLN}. One can generalize the Newton polygon method for finite order linear differential operators to linear differential operators like $L$ above and indeed obtain much more than what we proved. 
\end{remark}

\begin{theorem}\label{transhens} $\mathbb{T}_\Delta$ is $\sigma$-henselian.
\end{theorem}
\begin{proof} Since $\sigma(f(x))=f(x+1)$, exponentiality of $f$ and $\sigma(f)$ are the same, exercise 5.13 of ~\cite{vdHLN}. Also, $\mathbb T$ is the union of Hahn fields 
$$\mathbb{L}_0 \subseteq \mathbb{L}_1 \subseteq \mathbb{L}_2 \cdots$$
where $\mathbb{L}_{i+1}$ is obtained from $\mathbb{L}_i$ by taking exponentials, see ~\cite{vdHLN} page 98. Therefore $\sigma$ is an automorphism of each $\mathbb{L}_i$ and $\mathbb{T}$ is the directed union of the valued difference subfields $\mathbb{L}_i$. Since $K_w \subseteq \mathbb{L}_0$, the residue difference field of $\mathbb{L}_i$ is $\bk_w$ and by the above discussion $\mathbb{L}_i$ satisfies Axiom 2 for all $i$. Now, since each $\mathbb{L}_i$ is maximal, we can use Corrollary~\ref{crucialhens} to conclude that they are $\sigma$-henselian. Note that $\sigma$-henselianity is a universal-existential first-order property and such properties are preserved in unions of chains. Therefore $\mathbb{T}$ is $\sigma$-henselian. 
\end{proof}

Thus the results of the previous section are applicable, in particular:

\begin{corollary} $\mathbb{T}_\Delta$ is elementarily equivalent to the Hahn difference field $\bk_w((t^{\Gamma_w}))$ (equipped with its natural cross-section).
\end{corollary}

In this section we considered the particular automorphism $\sigma(f(x))=f(x+1)$ but indeed one can carry out the arguments above for any $\sigma$ which is given by right-composition with $g(x)$ where $g(x)$ is an infinite positive transseries of exponential and logarithmic depth zero. For that general case one still considers the coarsening above, the only difference would be to explicitly follow Remark~\ref{t1} instead of the discussion which precedes it.

\bibliography{bibt}

\end{document}